\documentclass{article}

\usepackage{amsmath}
\usepackage{amssymb}
\usepackage{amsthm}
\usepackage{enumerate}
\usepackage{cases}
\usepackage{stmaryrd}

\numberwithin{equation}{section}  

\newtheorem{lem}{Lemma}[section]
\newtheorem{thm}[lem]{Theorem}
\newtheorem{prp}[lem]{Proposition}
\newtheorem{cor}[lem]{Corollary}

\theoremstyle{definition}
\newtheorem{dfn}[lem]{Definition}

\theoremstyle{remark}
\newtheorem{exm}[lem]{Example}
\newtheorem{rmk}[lem]{Remark}

\newtheorem*{problem}{Problem}

\newcommand{\Rmnum}[1]{\uppercase\expandafter{\romannumeral #1}}

\newcommand{\ii}{\sqrt{-1}}   

\newcommand{\ba}{\bar{\alpha}}   
\newcommand{\bb}{\bar{\beta}}
\newcommand{\bg}{\bar{\gamma}}
\newcommand{\bl}{\bar{\lambda}}
\newcommand{\bm}{\bar{\mu}}
\newcommand{\bn}{\bar{\nu}}

\newcommand{\tth}{\tilde{\theta}}

\newcommand{\la}{\langle}
\newcommand{\ra}{\rangle}

\begin{document}

\title{Rigidity Theorems for Complete Sasakian Manifolds with Constant Pseudo-Hermitian Scalar Curvature\footnote{Supported by NSFC grant No. 11271071 and LMNS, Fudan} \footnote{Key words: Rigidity theorem, Chern-Moser tensor,  CR Yamabe invariant, Sasakian manifold, pseudo-Einstein manifold, pseudo-Hermitian space form}
 \footnote{2010 Mathematics Subject Classification. Primary: 32V05. Secondary: 32V20, 53C24, 53C25}}
\author{Yuxin, Dong \and Hezi, Lin \and Yibin, Ren\footnote{Corresponding author}}

\maketitle

\begin{abstract}
The orthogonal decomposition of the Webster curvature provides us a way to characterize some canonical metrics on a pseudo-Hermitian manifold.  We derive some subelliptic differential inequalities from the Weitzenb\"ock formulas for the traceless pseudo-Hermitian Ricci tensor and the Chern-Moser tensor of Sasakian manifolds with constant pseudo-Hermitian scalar curvature and Sasakian pseudo-Einstein manifolds respectively. By means of either subelliptic estimates or maximum principle, some rigidity theorems are established to characterize Sasakian pseudo-Einstein manifolds among Sasakian manifolds with constant pseudo-Hermitian scalar curvature and Sasakian space forms among Sasakian pseudo-Einstein manifolds respectively.
\end{abstract}

\section{Introduction}
A fundamental result in Riemannian geometry is that the curvature tensor $Rm$ of an $n$ dimensional Riemannian manifold $(M^n,g)$ can be decomposed into three mutually orthogonal irreducible components:
$$
Rm=W + \frac{1}{n-2} E \owedge g + \frac{\rho}{2n (n-1)} g \owedge g
$$
where $W$ denotes the Weyl conformal curvature tensor, $E$ and $\rho$ are the traceless part of the Ricci curvature and the scalar curvature respectively. The vanishing of some component in this decomposition characterizes some special metric on Riemannian manifolds.
One of the important problems in Riemannian geometry is to investigate the rigidity phenomena of some canonical metrics on Riemannian manifolds. As is known, Einstein manifolds play an important role in both geometry and physics, while real spaces forms are the simplest geometric models. In \cite{goldberg1980application,hebey1996effectivel,iton2002isolation,kim2011rigidity,pigola2008vanishing,shen1990some,singer1992positive,xu2010p}, some rigidity results were given to characterize Einstein manifolds and real space forms among suitably larger classes of Riemannian manifolds respectively. The key point of their methods is a special type of differential inequalities derived from a suitable Weitzenb\"ock formula and a refined Kato inequality, which enables one to use either the maximum principle or the elliptic estimates to derive the rigidity results, under either pointwise pinching conditions or global integral pinching conditions.

In 1974, S. S. Chern and J. K. Moser introduced the so-called Chern-Moser
tensor $C$ for a pseudo-Hermitian manifold, which plays the role analogous
to the Weyl curvature tensor for a Riemannian manifold (cf. \cite{chern1974real,webster1978pseudo}). Let $R$
be the Webster curvature of the pseudo-Hermitian manifold. Then we
have the following orthogonal decomposition:
\begin{align*}
R_{\ba \beta \lambda \bm} =&C_{\ba \beta \lambda \bm} + \frac{1}{n+2} (E_{\ba \beta} \delta_{\lambda \bm} + E_{\ba \lambda} \delta_{\beta \bm} + \delta_{\ba \beta} E_{\lambda \bm} + \delta_{\ba \lambda} E_{\beta \bm}) \\
& + \frac{\rho}{n(n+1)} (\delta_{\ba \beta} \delta_{\lambda \bm} + \delta_{\ba \lambda} \delta_{\beta \bm})
\end{align*}
where $C_{\overline{\alpha }\beta \lambda \overline{\mu }}$ is the
Chern-Moser tensor, $E_{\alpha \overline{\beta }}$ is the traceless
pseudo-Hermitian Ricci tensor and $\rho $ is the pseudo-Hermitian scalar
curvature. Recall that a pseudo-Hermitian manifold $(M, HM, J_b, \theta)$ is spherical if and
only if $C = 0$ ($dim \: M \geq 5$); $(M, HM, J_b, \theta)$ is pseudo-Einstein if and only if $E=0$.
Clearly $(M, HM, J_b, \theta)$ is of constant pseudo-Hermitian sectional curvature if and only if $C=0$ and $E=0$, provided it has constant pseudo-Hermitian scalar curvature.
On the other hand, a pseudo-Hermitian manifold is called
Sasakian if its pseudo-Hermitian torsion vanishes. As a special class of
pseudo-Hermitian manifolds, Sasakian manifolds have received much attention
over the past two decades, due to at least two reasons. Firstly, Sasakian
manifolds are the links of K\"ahler cones, and secondly, Sasakian pseudo-Einstein
manifolds play a special role in String theory (cf. \cite{boyer2008sasakian}).

In this paper, we investigate the following two rigidity problems for
complete Sasakian manifolds:
\begin{problem} \mbox{}\par
\begin{enumerate}[(i)]
\item the rigidity of Sasakian pseudo-Einstein
manifolds among Sasakian manifolds with constant pseudo-Hermitian scalar
curvature;  \label{a6}
\item  the rigidity of Sasakian space forms among
Sasakian pseudo-Einstein manifolds.  \label{a7}
\end{enumerate}
\end{problem}
For these purposes, we derive the
Weitzenb\"ock formulas for $|E|^2$ and $|C|^2$ respectively. It turns out
that the vanishing of the pseudo-Hermitian torsion and the constancy of $\rho $ imply that $E$
is a Codazzi type tensor, and if $(M, HM, J_b, \theta)$ is pseudo-Einstein, then $C$ satisfies
a Bianchi-type identity. These properties for $E$ and $C$, combined with some
refined Kato inequalities, enable us to deduce some differential
inequalities for $|E|$ and $|C|$ respectively. Note that the differential
inequalities for $|E|$ and $|C|$ are of the same kind and the main
differential operator appearing in them is the sub-Laplacian. We will treat
both compact and complete noncompact Sasakian manifolds, and use either the
maximum principle or the sub-elliptic estimates to derive the rigidity
results from these differential inequalities, under either pointwise
pinching conditions or global integral pinching conditions. Roughly
speaking, our results are as follows:
\begin{enumerate}
\item For Problem \eqref{a6}, under either
suitable pointwise pinching conditions or suitable $L^{n+1}$ integral
pinching conditions on  $|C|$ and $|E|$, the Sasakian manifolds become Sasakian
pseudo-Einstein manifolds (Section \ref{s1});
\item For Problem \eqref{a7}, under either suitable
pointwise pinching conditions or suitable $L^{n+1}$ integral pinching
conditions on $|C|$, the Sasakian pseudo-Einstein manifolds become Sasakian space
forms (Section \ref{s2}).
\end{enumerate}
These results may be regarded as the CR analogues to the rigidity
results for Riemannian manifolds mentioned previously. Actually the authors
in \cite{itoh2004isolation} considered the Problem \eqref{a7} for compact Sasakian manifold too. They
gave an integral-type rigidity theorem for the Chern-Moser tensor on compact Sasakian
pseudo-Einstein manifolds to characterize the odd dimensional spheres.
However, their integral norm for $C$ is not CR conformal invariant. Using a CR
Sobolev-type inequality, we may obtain a somewhat different result which seems to be more natural from the viewpoints of CR geometry. Besides, the main part
of this paper is to study the Problems \eqref{a6} and \eqref{a7} for complete noncompact
Sasakian manifolds. On the other hand, the authors in \cite{dong2014gap} have established
similar rigidity theorems for complete K\"ahler manifolds with constant
scalar curvature.

\section{Pseudo-Hermitian Geometry}
In this section, we present some basic notions and formulas on CR manifolds. For details, the readers may refer to \cite{boyer2008sasakian,dragomir2006differential,lee1988psuedo,webster1978pseudo}. Recall that a smooth manifold $M$ of real dimension ($2n+1$) is said to be a CR manifold if
there exists a smooth rank $n$ complex subbundle $T_{1,0} M \subset TM \otimes \mathbb{C}$ such that
$$
T_{1,0} M \cap T_{0,1} M =0
$$
and
$$
[\Gamma (T_{1,0} M), \Gamma (T_{1,0} M)] \subset \Gamma (T_{1,0} M)
$$
where $T_{0,1} M = \overline{T_{1,0} M}$ is the complex conjugate of $T_{1,0} M$.
Equivalently, the CR structure may also be described by the real subbundle $HM = Re \: \{ T_{1,0}M \oplus T_{0,1}M \}$ of $TM$ which carries a complex structure $J_b : HM \rightarrow HM$ defined by $J_b (X+\overline{X})= \ii (X-\overline{X})$ for any $X \in T_{1,0} M$.
Since $HM$ is naturally oriented by the complex structure, then $M$ is orientable if and only if
there exists a global nowhere vanishing 1-form $\theta$ such that $\theta (HM) =0 $.
Any such section $\theta$ is referred to as a pseudo-Hermitian structure on $M$. Given a pseudo-Hermitian structure $\theta$ on $M$, the Levi form $L_\theta $ is defined by
$$L_\theta (Z,\overline{W} ) = - \ii d \theta (Z, \overline{W})  $$
for any $Z , W \in T_{1, 0} M$.
\begin{dfn}
An orientable CR manifold $M$ with a pseudo-Hermitian structure $\theta$, denoted by $(M, HM, J_b, \theta)$, is called a pseudo-Hermitian manifold. A pseudo-Hermitian manifold $(M, HM, J_b, \theta)$  is said to be a strictly pseudoconvex CR manifold if its Levi form $L_\theta$ is positive definite.
\end{dfn}

From now on we assume that  $(M, HM, J_b, \theta)$  is a strictly pseudoconvex CR manifold. Therefore there exists a unique nowhere zero vector field $T$ transverse to $HM$, satisfying
$T \lrcorner \: \theta =1, \ T \lrcorner \: d \theta =0$. This vector field is called the characteristic direction of $(M, HM, J_b, \theta)$. We can extend $J_b$ to an endomorphism of $TM$ by requiring that $J_b T=0$.
Define the bilinear form $G_\theta$ by
$$
G_\theta (X, Y)= d \theta (X, J_b Y), \quad \mbox{for any }  \  X, Y \in HM.
$$
Since $G_\theta$ is $J_b $-invariant and  coincides with $L_\theta$ on $T_{1,0} M \otimes T_{0,1} M$, $G_\theta$ is also positive definite on $HM \otimes HM$. This allows us to define a Riemannian metric $g_\theta$ on $TM$ by
$$g_\theta (X, Y) = G_\theta (\pi_H X, \pi_H Y)+ \theta(X) \theta (Y), \quad X, Y \in  TM$$
where $\pi_H : TM \rightarrow HM$ is the projection with respect to the direct sum decomposition $TM = HM \oplus \mathbb{R} T$. This metric $g_\theta $ is usually called the Webster metric. Sometimes we denote it by $\la \cdot , \cdot \ra$.

\begin{dfn}
Assume that $(M, HM, J_b, \theta)$ and $( \widetilde{M}, \widetilde{HM}, \widetilde{J_b}, \widetilde{\theta} )$ are two strictly pseudoconvex CR manifolds. We say that $(M, HM, J_b, \theta)$ is CR conformal to $( \widetilde{M}, \widetilde{HM}, \widetilde{J_b}, \widetilde{\theta} )$ if there exists a diffeomorphism $f: M \rightarrow \widetilde{M}$ and a smooth function $u$ on $M$ such that
\begin{gather*}
df \circ J_b = \widetilde{J_b} \circ df \mbox{ and } \ f^* \tilde{\theta} = u \theta.
\end{gather*}
Furthermore $(M, HM, J_b, \theta)$ is called to be D-homothetic to $( \widetilde{M}, \widetilde{HM}, \widetilde{J_b}, \widetilde{\theta} )$ if in addition $u$ is constant.
\end{dfn}

On a strictly pseudoconvex CR manifold, there exists a canonical connection preserving the complex structure and the Webster metric.
\begin{prp} [\cite{webster1978pseudo}; cf. also \cite{dragomir2006differential}]
Let $(M, HM, J_b, \theta)$ be a strictly pseudoconvex CR manifold. Let $T$ be the characteristic direction and $g_\theta$ the Webster metric. Then there is a unique linear connection $\nabla$ on $M$ (called the Tanaka-Webster connection) such that:
\begin{enumerate}[(1)]
\item The Levi distribution HM is parallel with respect to $\nabla$;
\item $\nabla J_b=0$, $\nabla g_\theta=0$;
\item The torsion $T_\nabla$ of the connection $\nabla$ satisfies that for any $X, Y \in HM$,
$$
T_{\nabla} (X, Y)= 2 d \theta (X, Y) T  \mbox{ and }T_{\nabla} (T, J_b X) + J_b T_{\nabla} (T, X) =0.
$$
\end{enumerate}
\end{prp}

The pseudo-Hermitian torsion, denoted by $\tau$, is the $TM$-valued 1-form defined by $\tau(X) = T_{\nabla} (T,X)$.
A pseudo-Hermitian manifold is called Sasakian if $\tau \equiv 0$. Clearly, the D-homothetic transformation of a Sasakian manifold is also a Sasakian manifold.

Let $(M, HM, J_b, \theta)$ be a strictly pseudoconvex CR manifold of dimension $2n+1$. Let $\{ \eta_1, \dots, \eta_n \}$ be a local orthonormal frame of $T_{1,0} M$ defined on an open set $U \subset M$ , and $\{ \theta^1, \dots \theta^n \}$ its dual coframe. Note that $\tau(T_{1,0} M) \subset T_{0,1} M$ and $\tau$ is $g_\theta$-symmetric (cf. \cite{dragomir2006differential}). Thus we can write $\tau \eta_\alpha = A_{\alpha}^{\bb} \eta_{\bb}$ for some local smooth functions $A_{\alpha}^{\bb} : U \rightarrow \mathbb{C}$.
Then the structure equations are given by
\begin{numcases}{ }
d \theta = 2 \ii \theta^\alpha \wedge \theta^{\ba} , \\
d \theta^\alpha = \theta^\beta \wedge \theta^\alpha_\beta + \theta \wedge \tau^\alpha ,   \\
\theta^\alpha_\beta + \theta^{\bb}_{\ba}= 0 ,  \\
d \theta^\alpha_\beta = \theta^\gamma_\beta \wedge \theta^\alpha_\gamma + \Pi^\alpha_\beta
\end{numcases}
where $\theta^\alpha_\beta$'s are the Tanaka-Webster connection 1-forms with respect to $\{ \eta_\alpha \}$ and $\tau^\alpha= A^\alpha_{\bb} \theta^{\bb}$. In \cite{webster1978pseudo}, S. M. Webster showed that
\begin{align}
\Pi_\beta^\alpha= 2 \ii (\theta^\alpha \wedge \tau_\beta + \theta_\beta \wedge \tau^\alpha) + R^\alpha_{\beta \lambda \bm} \theta^\lambda \wedge \theta^{\bm} + A^\alpha_{\bm, \beta} \theta \wedge \theta^{\bm} - A_{\mu \beta ,}^{\quad \; \alpha} \theta \wedge \theta^\mu \label{a8}
\end{align}
where $R^\alpha_{\beta \lambda \bm}$ is called the Webster curvature.
He also derived the first Bianchi identity, i.e.  $R_{\ba \beta \lambda \bm} = R_{\ba \lambda \beta \bm}$. So the pseudo-Hermitian Ricci curvature can be defined by $R_{\lambda \bm}= R_{\ba \alpha \lambda \bm}$ and then the pseudo-Hermitian scalar curvature is $\rho= R_{\alpha \ba}=R_{\bb \beta \alpha \ba}$.  If $R_{\lambda \bm} = \frac{\rho}{n} \delta_{\lambda \bm}$, $(M, HM, J_b, \theta)$ is called a pseudo-Einstein manifold. To characterize the pseudo-Einstein manifolds, the traceless pseudo-Hermitian Ricci tensor is defined by $E_{\alpha \bb} = E_{\bb \alpha} = R_{\alpha \bb} - \frac{\rho}{n}$. Clearly $(M, HM, J_b, \theta)$ is pseudo-Einstein if and only if $E =0$. From \eqref{a8}, the pseudo-Hermitian torsion and its covariant derivative reflect partial information about the curvature tensor of the Tanaka-Webster connection. The second Bianchi identities were given by J. M. Lee in Lemma 2.2 of \cite{lee1988psuedo}.
\begin{lem}[Second Bianchi identities, \cite{lee1988psuedo}]
The Webster curvature satisfies the following identities
\begin{gather}
R_{\ba \beta \lambda \bm, \gamma} - R_{\ba \beta \gamma \bm, \lambda} = 2 \ii ( A_{\beta \gamma, \bm} \delta_{\ba \lambda} +  A_{\gamma \beta , \ba} \delta_{\lambda \bm} - A_{\beta \lambda, \bm} \delta_{\ba \gamma} -  A_{\lambda \beta, \ba} \delta_{\gamma \bm} ), \label{a4} \\
R_{\ba \beta \lambda \bm, 0}= A_{\lambda \beta, \ba \bm} + A_{\ba \bm , \beta \lambda} + 2 \ii (A_{\ba \bg} A_{\gamma \lambda } \delta_{\beta \bm} - A_{\beta \gamma} A_{\bg \bm} \delta_{\ba \lambda} ), \label{a5}
\end{gather}
and the contracted identities:
\begin{gather}
R_{\lambda \bm, \gamma} - R_{\gamma \bm, \lambda} = 2\ii (A_{\gamma \alpha, \ba} \delta_{\lambda \bm} - A_{\lambda \alpha, \ba} \delta_{\gamma \bm}), \label{a1} \\
\rho_\lambda - R_{\lambda \bm, \mu} = 2 \ii (n-1) A_{\lambda \mu, \bm}, \label{a9} \\
R_{\lambda \bm, 0} = A_{\lambda \alpha, \ba \bm} + A_{\ba \bm, \alpha \lambda}, \label{a2} \\
\rho_0 =A_{\lambda \alpha, \ba \bl} + A_{\ba \bl, \alpha \lambda}. \label{a3}
\end{gather}
\end{lem}

In \cite{webster1978pseudo}, S. M. Webster introduced a pseudo-Hermitian analogue of the notion of holomorphic sectional curvature in Hermitian geometry. For consistency, we recall it in the terminology of \cite{dragomir2006differential}. Given $x \in M$, let $G_1 (HM)_x$ consist of all 2-planes $\sigma = span \{ X, J_b X\} \subset T_x M$ for all $X \in H_x M$. Then $G_1 (HM)$ (the disjoint union of all $G_1 (HM)_x $) is a fiber bundle over $M$ with the standard fiber $\mathbb{C} P^{n-1}$. Define a function
$
k_\theta: G_1 (HM) \rightarrow \mathbb{R}
$
by
\begin{align*}
k_\theta (\sigma ) =  \frac{1}{4} \frac{ g_\theta \big( R (X, J_b X) J_b X, X \big) }{g_\theta (X, X)^2}
\end{align*}
for any $\sigma \in G_1 (HM) $ spanned by $X$ and $J_b X$. Actually if $Z = \frac{1}{\sqrt{2}} (X - \ii J_b X)$, then
\begin{align*}
k_\theta (\sigma) = - \frac{1}{4} \frac{ g_\theta ( R(Z, \overline{Z}) \overline{Z}, Z ) }{ g_{\theta} ( Z, \overline{Z} )^2 }
\end{align*}
This function $k_\theta$ is referred to as the pseudo-Hermitian sectional curvature of $(M, HM, J_b, \theta)$.
Clearly if $R_{\ba \beta \lambda \bm} = 2 \kappa (\delta_{\ba \beta} \delta_{\lambda \bm} + \delta_{\ba \lambda} \delta_{\beta \bm} ) $ for some constant $\kappa$, then $(M, HM, J_b , \theta)$ has constant pseudo-Hermitian sectional curvature $ \kappa $.
As a consequence of Theorem 1.6 in \cite{dragomir2006differential}, the Riemannian curvature $R^\theta$ and the Webster curvature $R$ of a Sasakian manifold have the following connection:
for any $X, Y, Z \in T M$,
\begin{align}
R^\theta (X, Y) Z = & R(X, Y)Z + g_\theta (JX, Z) JY- g_\theta (JY, Z) JX + 2 d \theta (X, Y) JZ  \nonumber \\
& + \theta (X) g_\theta (Y, Z) T - \theta (Y) g_\theta (X, Z) T - \theta(Z) \theta(X) Y + \theta(Z) \theta(Y) X \label{a11}
\end{align}
which gives the relationship between the Riemannian Ricci curvature $Ric^\theta$ and the pseudo-Hermitian Ricci curvature:
\begin{align}
Ric^\theta_{\alpha \bb} = R_{\alpha \bb} - 2 \delta_{\alpha \bb}, \quad Ric^\theta_{\alpha \beta} = Ric^\theta_{\alpha 0}=0, \quad \mbox{and } Ric^\theta_{00} = 2n.  \label{a10}
\end{align}
From \eqref{a11}, if a Sasakian manifold has constant pseudo-Hermitian sectional curvature $\kappa $, it has constant $J_b$-sectional curvature $ 4 \kappa -3$ in the terminology of \cite{blair2010riemannian,boyer2008sasakian}. Such manifold is called Sasakian space form. The following theorem of Tanno \cite{tanno1969sasakian} gives the classification of Sasakian space forms.

\begin{lem}[\cite{tanno1969sasakian}; also cf. {\cite[p.\,142]{blair2010riemannian}} and  {\cite[p.\,229]{boyer2008sasakian}}] \label{alm4}
Let $(M, HM, J_b, \theta)$ be a complete simply connected Sasakian manifold with constant pseudo-Hermitian sectional curvature $\kappa$. Then $(M, HM, J_b, \theta)$ is  one of the following:
\begin{enumerate}
\item if $\kappa >0 $, $(M, HM, J_b, \theta)$ is D-homothetic to $S^{2n+1}$;
\item if $\kappa =0 $, $(M, HM, J_b, \theta)$ is D-homothetic to $\mathbb{H}^n$;
\item if $\kappa <0 $, $(M, HM, J_b, \theta)$ is D-homothetic to $B^n_{\mathbb{C}} \times \mathbb{R}$.
\end{enumerate}
\end{lem}

By studying the CR conformal transformations, the authors in \cite{chern1974real,webster1978pseudo} found an important tensor called Chern-Moser tensor, which plays the role analogous to Weyl tensor in Riemannian geometry. It is defined by
\begin{align}
C_{\ba \beta \lambda \bm} =& R_{\ba \beta \lambda \bm} - \frac{1}{n+2} (E_{\ba \beta} \delta_{\lambda \bm} + E_{\ba \lambda} \delta_{\beta \bm} + \delta_{\ba \beta} E_{\lambda \bm} + \delta_{\ba \lambda} E_{\beta \bm}) \nonumber \\
&  - \frac{\rho}{n(n+1)} (\delta_{\ba \beta} \delta_{\lambda \bm} + \delta_{\ba \lambda} \delta_{\beta \bm}) \label{c3}
\end{align}
which is the projection of the Webster curvature into the traceless subspace.
Moreover under a CR conformal transformation $\tilde{\theta} = e^{2u} \theta$ for some $u \in C^{\infty} (M)$, the new Chern-Moser tensor and the original one differ by a conformal factor, i.e. $\tilde{C} = e^{-2u} C$. Hence
$$
||C||_{L^{n+1} (M)} = \bigg( \int_M |C|^{n+1} \theta \wedge (d \theta)^n \bigg)^{\frac{1}{n+1}}
$$
is a CR conformal invariant. It is known that if the Chern-Moser tensor vanishes and $2n+1 \geq 5$, the strictly pseudoconvex CR manifold is locally CR isomorphic to the unit sphere $S^{2n+1} \subset \mathbb{C}^{n+1}$. Such manifold is called spherical CR manifold. Clearly a Sasakian manifold is a Sasakian space form if and only if it is spherical and pseudo-Einstein. In Sasakian geometry, the Chern-Moser tensor agrees with the contact Bochner curvature tensor defined in \cite{itoh2004isolation}.

Since the scalar curvature of an Einstein manifold is constant, it is natural to ask whether this property is still true for  the pseudo-Hermitian scalar curvature of a pseudo-Einstein manifold. We may show that the answer is affirmative for some special case.
\begin{lem} \label{alm2}
Let $(M, HM, J_b, \theta)$ be a pseudo-Einstein manifold of dimension $2n+1 \geq 5$. If the divergence of the pseudo-Hermitian torsion vanishes, the pseudo-Hermitian scalar curvature is constant.
\end{lem}
\begin{proof}
Since $(M, HM, J_b, \theta)$ is pseudo-Einstein, then $R_{\lambda \bm} = \frac{\rho}{n} \delta_{\lambda \bm}$. Since the divergence of the pseudo-Hermitian torsion vanishes and $n \geq 2$, by \eqref{a9}, we have
\begin{align*}
0 = \rho_\lambda - R_{\lambda \bm, \mu} = \rho_{\lambda } - \frac{1}{n} \rho_{\lambda}
\end{align*}
which implies $\rho_{\lambda} =0 $. Moreover \eqref{a3} yields $\rho_0 =0$. Hence $\rho$ is constant.
\end{proof}

However, the pseudo-Einstein condition does not imply the constancy
of the pseudo-Hermitian scalar curvature in general, even if the Chern-Moser
tensor vanishes. Before giving a counter-example, let us recall the following lemma.
\begin{lem}[\cite{lee1988psuedo}] \label{alm1}
Let $(M, HM, J_b, \theta)$ be a strictly pseudoconvex CR manifold. Under the CR conformal transformation $\tilde{\theta} = e^{2u} \theta$ for some $u \in C^{\infty} (M)$, the pseudo-Hermitian torsion and the pseudo-Hermitian curvature transform as follows:
\begin{gather*}
\tilde{A}_{\alpha \beta } = e^{-2u} \big( A_{\alpha \beta} + \ii u_{\alpha \beta}   - 2 \ii u_{\alpha } u_{\beta}   \big)
\end{gather*}
and
\begin{gather*}
e^{2u} \tilde{R}_{\lambda \bm} = R_{\lambda \bm} - (n+2) ( u_{\lambda \bm} + u_{\bm \lambda} ) -\delta_{\lambda \bm} \big(  u_{\alpha \ba} + u_{\ba \alpha} + 4 (n+1) u_\alpha u_{\ba} \big) \\
e^{2u} \tilde{\rho} = \rho - 2 (n+1) (  u_{\alpha \ba} + u_{\ba \alpha}) - 4 (n+1) u_\alpha u_{\ba}
\end{gather*}
where the covariant derivatives of $u$ are computed with respect to the original pseudo-Hermitian structure $(M, HM, J_b, \theta)$ and $\tilde{A}_{\lambda \mu}$, $\tilde{R}_{\lambda \bm}$ are evaluated with respect to the new coframe $\tilde{\theta}^\alpha = e^u (\theta^\alpha + \ii u^\alpha \theta ) $.
\end{lem}

\begin{exm} \label{aex1}
Let $( \mathbb{H}^n, H \mathbb{H}^n, J_b, \theta )$ be the Heisenberg group. Under the standard coordinates $ (z, t) $ in $ \mathbb{H}^n \cong \mathbb{C}^n \times \mathbb{R} $, we have
\begin{align*}
\theta = d t + \sum_{\alpha=1}^n \ii \big( z^\alpha d \bar{z}^\alpha - \bar{z}^\alpha d z^\alpha  \big) , \mbox{ and } \  d \theta = 2 \ii d z^\alpha \wedge d \bar{z}^\alpha.
\end{align*}
One may choose the following frame and coframe fields
\begin{align*}
\eta_\alpha = \frac{ \partial }{\partial z^\alpha} + \ii \bar{z}^\alpha \frac{ \partial }{ \partial t }, \quad \mbox{and } \ \theta^\alpha = d z^\alpha.
\end{align*}
Since $d \theta^\alpha = d d \bar{z}^\alpha =0$, the connection coefficients $ \theta^\alpha_\beta =0 $ and the pseudo-Hermitian torsion $ \tau =0$. It is easy to verify that the pseudo-Hermitian Ricci curvature, the pseudo-Hermitian scalar curvature
and the Chern-Moser tensor of the Heisenberg group all vanish

Now we consider its CR conformal transformation $\tth = e^{2u} \theta$ with $u(z, t) = |z|^2 $. By Lemma \ref{alm1}, we have
\begin{align*}
\tilde{A}_{\lambda \mu} = -2 \ii \bar{z}^{\lambda} \bar{z}^{\mu} e^{-2 |z|^2 }
\end{align*}
and
\begin{gather*}
\tilde{R}_{\lambda \bm} =-4 (n+1) (1+ |z|^2 ) e^{-2 |z|^2}  \delta_{\lambda \bm}, \\
\tilde{\rho} = -4 n (n+1) (1+ |z|^2 ) e^{-2 |z|^2}.
\end{gather*}
Hence the pseudo-Hermitian manifold $( \mathbb{H}^n, H \mathbb{H}^n, J_b, \tth )$ is pseudo-Einstein too, but non-Sasakian. Since the Chern-Moser tensor changes conformally under the CR conformal transformation, the new Chern-Moser tensor also vanishes. However the pseudo-Hermitian scalar curvature of $\tth $ is nonconstant.
\end{exm}

A useful tool in Riemannian geometry is the Ricci identity for commuting covariant derivatives.
Naturally we need a similar formula for the Tanaka-Webster connection.
\begin{lem}[Ricci Identities]
\begin{align}
R_{\alpha \bb, \lambda \bm} - R_{\alpha \bb, \bm \lambda}= R_{\gamma \bb} R_{\bg \alpha \lambda \bm} + R_{\alpha \bg} R_{\gamma \bb \lambda \bm} + 2 \ii \delta_{\lambda \bm } R_{\alpha \bb, 0} \label{b1}
\end{align}
and
\begin{align}
C_{\ba \beta \varsigma \bm, \lambda \bg} -& C_{\ba \beta \varsigma \bm, \bg \lambda} = C_{\bn \beta \varsigma \bm} R_{\nu \ba  \lambda \bg} + C_{\ba \nu \varsigma \bm} R_{\bn \beta  \lambda \bg}  \nonumber \\
& + C_{\ba \beta \nu \bm} R_{\bn \varsigma  \lambda \bg} + C_{\ba \beta \varsigma \bn} R_{\nu \bm  \lambda \bg} + 2 \ii \delta_{\lambda \bg } C_{\ba \beta \varsigma \bm, 0 } . \label{d10}
\end{align}
\end{lem}
The above results are consequences of the following general identity: for any $ \sigma  \in \Gamma ( \otimes^q T^* M ) $ and $X, \: Y, \: X_1, \dots, X_q \in \Gamma ( TM ) $, we have
\begin{align*}
& \big( \nabla^2 \sigma \big) (X_1, \dots, X_p; Y, X) - \big( \nabla^2 \sigma \big) (X_1, \dots, X_p; X, Y) \\
& = - \big( R(Y, X)  \sigma \big) (X_1, \dots, X_p) + \big( \nabla_{T_{\nabla} (Y, X)} \sigma  \big) (X_1, \dots, X_p) \\
& =  \sum_{i=1}^q \sigma \big(X_1, \dots,  R (Y, X) X_i, \dots, X_q \big) + \big( \nabla_{T_{\nabla} (Y, X)} \sigma  \big) (X_1, \dots, X_p) .
\end{align*}

Analogous to the Laplace operator in Riemannian geometry, there is a degenerate elliptic operator in CR geometry which is called sub-Laplace operator. Let $(M, HM, J_b, \theta)$ be a strictly pseudoconvex CR manifold of dimension $2n+1$.  The nowhere vanishing form $\theta \wedge (d \theta)^n$ defines a volume form of $M$. Hence the divergence $div (X)$ of a vector field $X$  is given by
\begin{align*}
\mathcal{L}_X \theta \wedge (d \theta)^n = div (X) \theta \wedge (d \theta)^n
\end{align*}
where $\mathcal{L}$ denotes the Lie derivative.
\begin{dfn}
The sub-Laplace operator is the differential operator $\triangle_b$ defined by
\begin{align*}
\triangle_b u =  div (\nabla_b u), \quad \mbox{for any } \ u \in C^\infty (M)
\end{align*}
where $\nabla_b u = \pi_H \nabla u$ is the horizontal part of $\nabla u$.
\end{dfn}

\begin{lem}[\cite{dragomir2006differential}]
Under the above notions, we have
\begin{align*}
\triangle_b u = u_{\alpha \ba} + u_{\ba \alpha}, \quad \mbox{for any } \ u \in C^\infty (M).
\end{align*}
\end{lem}

The sub-Laplace operator enjoys a similar maximum principle as the Laplace operator.

\begin{lem}[Bony's Maximum Principle \cite{bony1969principe}; also cf. \cite{jost1998subelliptic}] \label{alm3}
Let $(M, HM, J_b, \theta)$ be a strictly pseudoconvex CR manifold and $\Omega $ an open set of $M$. Suppose that $u \in C^2 (\Omega) $ satisfies $\triangle_b u \geq 0$. If there exists $x_0 \in \Omega$ such that $ 0 \leq u(x_0) = \sup_\Omega u < + \infty $, then $u \equiv u(x_0)$ on $\Omega$. In particular, if $M$ is compact, then $u$ must be constant.
\end{lem}

Since $(M, g_\theta)$ is also a Riemannian manifold, we denote by $\triangle$ its Laplace operator, $r$ the Riemannian distance function to a fixed point $x_0 \in M$ and $B_r$  the ball of radius $r$ centered at $x_0$.

\section{A general gap theorems for subelliptic differential inequalities}
In this section, we establish a general gap result for a special class of subelliptic differential inequalities on a complete pseudo-Hermitian manifold; this  result will be applied to  get some rigidity theorems of Sasakian pseudo-Einstein manifolds and Sasakian space forms later. Our method is basically a partial process of the Moser's iteration which is similar to that one used in \cite{pigola2008vanishing} for some elliptic differential inequalities.

\begin{prp} \label{blm1-1}
Let $(M, HM ,J_b, \theta)$ be a complete noncompact strictly pseudoconvex CR manifold and assume that, for some $0 <  \alpha < 1$ and some function h, the CR Sobolev-type inequality
\begin{align}
\int_M ( |\nabla_b \varphi|^2 + h \varphi^2) \geq S(\alpha)^{-1} \left( \int_M |\varphi|^{\frac{2}{1-\alpha}} \right)^{1-\alpha} \label{f6-1}
\end{align}
holds for every $\varphi \in C_c^\infty (M)$ with a positive constant $S(\alpha ) >0$. Suppose that $0 \neq \psi \in Lip_{loc} (M)$ is a nonnegative solution of
\begin{align}
\psi \triangle_b \psi + a(x) \psi^2 + B |\nabla_b \psi|^2 \geq 0 \quad \mbox{ weakly on} \ M  \label{f7-1}
\end{align}
satisfying
\begin{align}
\int_{B_r} |\psi|^\sigma = o (r^2) \quad as \quad r \rightarrow +\infty  \label{f8-1}
\end{align}
with $B \in \mathbb{R}, \ \sigma -B -1 >0, \ \sigma \geq 2$ , and $a(x) \in C^0(M)$. Then for any $\delta >0$ and $0 < \epsilon < \sigma -B -1$, we have
\begin{align}
\left| \left|   \left( a (x) +C_{\delta, \epsilon} h \right)_+ \right| \right|_{L^{\frac{1}{\alpha}} (M)} \geq C_{\delta, \epsilon} S(\alpha)^{-1}  \label{f9-1}
\end{align}
where $C_{\delta, \epsilon} = 4 \sigma^{-2} (1+ \delta)^{-1} (\sigma -B -1 - \epsilon)$.
\end{prp}

\begin{proof}
By \eqref{f7-1}, we know that for any nonnegative test function $\chi \in W^{1,2}_c (M)$,
\begin{align*}
\int_M a \psi^2 \chi \geq  \int_M \psi \la \nabla_b \psi, \nabla_b \chi \ra + (1-B) \chi |\nabla_b \psi|^2 .
\end{align*}
Let $\phi_r \in C^\infty_c (B_{2r})$ be a family of cutoff functions with
\begin{align*}
\phi_r =1 \quad \mbox{on} \ B_r, \quad \mbox{and} \  |\nabla \phi_r| \leq \frac{2}{r} \quad \mbox{on} \   M
\end{align*}
which implies $|\nabla_b \phi_r| \leq |\nabla \phi_r| \leq \frac{2}{r} $ on $M$. Fix $\eta >0 $. Then the test function $\chi= (\psi+\eta)^{\sigma-2} \: \phi^2_r$ is a Lipschitz function. Applying the CR Sobolev-type inequality \eqref{f6-1} to it, we derive
\begin{align*}
\int_M a \psi^2 (\psi+ \eta)^{\sigma-2} \phi^2_r &\geq \int_M  2 (\psi+\eta)^{\sigma -2} \phi_r \psi \la \nabla_b \psi, \nabla_b \phi_r  \ra  \\
& \quad + \int_M  \left[(1-B) + (\sigma -2) \frac{\psi}{\psi + \eta} \right] (\psi+\eta)^{\sigma -2} \phi^2_r |\nabla_b \psi|^2 .
\end{align*}
By Lebesgue dominated convergence theorem and letting $\eta \rightarrow 0$, the above inequality becomes
\begin{align*}
\int_M a \psi^{\sigma} \phi_r^2 &\geq \int_M  2 \psi^{\sigma -1} \phi_r\la \nabla_b \psi, \nabla_b \phi_r  \ra + \int_M  [(1-B) + (\sigma -2) ] \psi^{\sigma -2} \phi_r^2 |\nabla_b \psi|^2 .
\end{align*}
Using Cauchy-Schwarz inequality, we have
\begin{align*}
 (\sigma -B -1 - \epsilon)  \int_M \phi_r^2 \psi^{\sigma -2} |\nabla_b \psi|^2 \leq \int_M a \psi^{\sigma} \phi_r^2 + \frac{1}{\epsilon} \int_M \psi^{\sigma } |\nabla_b \phi_r|^2 .
\end{align*}
When $\sigma \geq 2$,  $\psi^{\frac{\sigma}{2}} \phi_r$ is a Lipschitz function. Applying the CR Sobolev-type inequality \eqref{f6-1} to it, we obtain
\begin{align}
& S(\alpha)^{-1} \left( \int_M (\psi^{\frac{\sigma}{2}} \phi_r )^{\frac{2}{1- \alpha}} \right)^{1- \alpha} \leq \int_M |\nabla_b ( \psi^{\frac{\sigma}{2}} \phi_r) |^2  + h \phi_r^2 \psi ^{\sigma} \nonumber \\
& \quad \leq (1 + \delta) \frac{\sigma^2}{4} \int_M \psi^{\sigma -2} \phi_r^2 |\nabla_b \psi|^2 + (1+ \frac{1}{\delta}) \int_M \psi^{\sigma} | \nabla_b \phi_r|^2 + \int_M  h \phi_r^2 \psi^{\sigma} \nonumber \\
& \quad \leq C_{\delta, \epsilon}^{-1} \int_M (a + h C_{\delta, \epsilon})_+ \psi^{\sigma} \phi_r^2 + ( C_{\delta, \epsilon}^{-1} \epsilon^{-1} +1+ \frac{1}{\delta}) \int_M \psi^{\sigma} | \nabla_b \phi_r|^2 \label{f1}
\end{align}
where $(a + h C_{\delta, \epsilon})_+$ is the nonnegative part of $a + h C_{\delta, \epsilon}$ and
\begin{align*}
C_{\delta, \epsilon} = \frac{4}{\sigma^2} \frac{\sigma -B -1 - \epsilon}{ 1+ \delta} .
\end{align*}
But H\"older's inequality implies that
\begin{align*}
\int_M (a + h C_{\delta, \epsilon})_+ \psi^{\sigma} \phi_r^2 \leq || (a + h C_{\delta, \epsilon} )_+ ||_{L^{\frac{1}{\alpha}} (B_{2r})} \left( \int_M (\psi^{\frac{\sigma}{2}} \phi_r )^{\frac{2}{1- \alpha}} \right)^{1- \alpha} .
\end{align*}
Substituting the above inequality to \eqref{f1}, then
\begin{align*}
& \left[ S(\alpha)^{-1} - C_{\delta, \epsilon}^{-1} || (a + h C_{\delta, \epsilon})_+ ||_{L^{\frac{1}{\alpha}} (B_{2r})}  \right] \left( \int_M (\psi^{\frac{\sigma}{2}} \phi_r )^{\frac{2}{1- \alpha}} \right)^{1- \alpha} \\
& \qquad  \leq ( C_{\delta, \epsilon}^{-1} \epsilon^{-1} +1+ \frac{1}{\delta}) \int_M \psi^{\sigma} | \nabla_b \phi_r|^2 \leq ( C_{\delta, \epsilon}^{-1} \epsilon^{-1} +1+ \frac{1}{\delta})  \frac{4}{r^2} \int_{B_{2r}} \psi^{\sigma} .
\end{align*}
Since $\psi \neq 0$, we let $r \rightarrow + \infty$ and use the assumption \eqref{f8-1} to find
\begin{align*}
|| (a + h C_{\delta, \epsilon})_+ ||_{L^{\frac{1}{\alpha}} (M)} \geq S(\alpha)^{-1}  C_{\delta, \epsilon} .
\end{align*}
\end{proof}

\begin{rmk}
One can also obtain a similar  result for $0< \sigma < 2$ by using the method in \cite{pigola2008vanishing}. Since we do not need it in this paper, we omit its details here.
\end{rmk}

In \cite{jerison1987yamabe},  D. Jerison and J. M. Lee introduced the following CR Yamabe constant
\begin{align}
\lambda(M) = \inf_{ \substack{ 0 \neq  u  \in  C^{\infty} (M) \\ supp \: u \Subset  M } } \frac{\int_M ( b_n|\nabla_b u|^2 + \rho u^2 ) \theta \wedge (d \theta)^n }{(\int_M |u|^p \theta \wedge (d \theta)^n  )^{\frac{2}{p}} } \label{f2}
\end{align}
where $b_n=p=2 + \frac{2}{n}$.
Hence if $\lambda(M)$ is positive, it provides a class of CR Sobolev-type inequality \eqref{f6-1} with
$$
h = \frac{n}{2n+2} \rho, \ \alpha = \frac{1}{n+1} \mbox{ and }  S(\alpha) = (2+ \frac{2}{n}) \lambda (M)^{-1}.
$$
It is known that the CR Yamabe constant of Heisenberg group and the odd dimension sphere are positive (cf. \cite{dragomir2006differential,jerison1987yamabe}). Combining \eqref{f2} with Proposition \ref{blm1-1}, we have the following corollary.

\begin{cor} \label{blm1}
Let $(M, HM ,J_b, \theta)$ be a complete noncompact strictly pseudoconvex CR manifold with positive CR Yamabe constant. Suppose that $\psi \in Lip_{loc} (M)$ is a nonnegative solution of
\begin{align}
\psi \triangle_b \psi + a(x) \psi^2 + B |\nabla_b \psi|^2 \geq 0 \quad \mbox{ weakly on} \ M  \label{f7}
\end{align}
satisfying
\begin{align}
\int_{B_r} |\psi|^\sigma = o (r^2) \quad as \quad r \rightarrow +\infty  \label{f8}
\end{align}
with $B \in \mathbb{R}, \ \sigma -B -1 >0, \ \sigma \geq 2$ , and $a(x) \in C^0(M)$. If there are $\delta >0$ and $0 < \epsilon < \sigma -B -1$ such that
\begin{align}
\left| \left|   \left( a (x) + \tilde{C}_{\delta, \epsilon} \rho \right)_+ \right| \right|_{L^{n+1 } (M)} < \tilde{C}_{\delta, \epsilon} \lambda (M)  \label{f9}
\end{align}
where $\tilde{C}_{\delta, \epsilon} = \frac{2n}{n+1} \frac{\sigma -B -1 -\epsilon }{\sigma^2 (1+ \delta)}  $, then $\psi \equiv 0$.
\end{cor}

\section{Rigidity Theorems of Pseudo-Einstein manifolds} \label{s1}
In this section, we consider a complete Sasakian manifold with
constant pseudo-Hermitian scalar curvature. First, under some suitable $L^p$ conditions or pinching conditions on the Chern-Moser tensor and
the traceless  pseudo-Hermitian Ricci tensor, we prove that such Sasakian manifolds must be pseudo-Einstein.
Second, we use the maximum principle
to prove that when the Chern-Moser tensor and the traceless pseudo-Hermitian Ricci tensor
satisfy some $L^{\infty}$ pinching condition, then $(M, HM, J_b, \theta )$ must be pseudo-Einstein too.
Finally, we give a simple proof to show that if a compact Sasakian manifold has constant pseudo-Hermitian scalar curvature and quasi-positive orthogonal pseudo-Hermitian sectional curvature, then it is pseudo-Einstein.

Let $(M, HM, J_b, \theta)$ be a $(2n+1)$-Sasakian manifold with constant pseudo-Hermitian scalar curvature. Since the pseudo-Hermitian torsion vanishes, the contracted identity \eqref{a2} implies that $R_{\alpha \bb,0}=0$. Hence \eqref{b1} yields
\begin{align}
E_{\alpha \bb, \lambda \bm} - E_{\alpha \bb, \bm \lambda}= E_{\gamma \bb} R_{\bg \alpha \lambda \bm} + E_{\alpha \bg} R_{\gamma \bb \lambda \bm}.
\end{align}
The Codazzi equation, i.e. $E_{\alpha \bb, \gamma} = E_{\gamma \bb, \alpha}$, follows from \eqref{a1} and the constancy of the pseudo-Hermitian scalar curvature. By direct calculations, we have
\begin{align}
\frac{1}{2} \triangle_b |E_{\alpha \bb}|^2 = 2 E_{\alpha \bb, \gamma} E_{\ba \beta, \bg} + E_{\alpha \bb, \gamma \bg} E_{\ba \beta} + \overline{E_{\alpha \bb, \gamma \bg} E_{\ba \beta}} \label{c1}
\end{align}
and
\begin{align}
E_{\alpha \bb, \gamma \bg} &= E_{\gamma \bb, \alpha \bg} = E_{\gamma \bb, \bg \alpha} + E_{\lambda \bb} R_{\bl \gamma \alpha \bg} + E_{\gamma \bl } R_{\lambda \bb \alpha \bg} \nonumber \\
&= E_{\lambda \bb} R_{\alpha \bl} + E_{\gamma \bl } R_{\lambda \bb \alpha \bg} .  \label{c2}
\end{align}
Substituting \eqref{c2} and \eqref{c3} into \eqref{c1}, we find
\begin{align}
\frac{1}{2} \triangle_b |E|^2
=  |\nabla_b E|^2  + \frac{2n}{n+2} tr_{G_\theta} E^3  - 4 E_{\gamma \bl } C_{ \bb \lambda \alpha \bg} E_{\ba \beta}
 + \frac{2 \rho}{n+1} |E|^2  \label{c4}
\end{align}
where $\nabla_b E$ is the horizontal part of $\nabla E$ and $tr_{G_\theta} E^3 = 2 E_{\alpha \bl} E_{\lambda \bb} E_{\ba \beta}$.
We use the method in \cite{hebey1996effectivel} to obtain the Kato inequality of the traceless pseudo-Hermitian Ricci tensor $E$.

\begin{lem}
If $(M, HM, J, \theta)$ is a Sasakian manifold of dimension $ 2 n+1 \geq 5 $ and with constant pseudo-Hermitian scalar curvature, then
\begin{align}
\frac{1}{4} |\nabla_b | E|^2 |^2 \leq \frac{n}{n+1} |E|^2 |\nabla_b E|^2 . \label{c5}
\end{align}
\end{lem}

\begin{proof}
Since $(E_{\alpha \bb})$ is a Hermitian matrix, we can choose some proper orthonormal basis $\{ \theta^\alpha \}$ such that $(E_{\alpha \bb})$ is diagonal at a given point and assume that $\{ \lambda_\alpha \}^n_{\alpha =1} $ are the eigenvalues.
Hence we can calculate that
\begin{align*}
\big|\nabla_b |E|^2 \big|^2 = 32 \sum_\gamma \big| \sum_{\alpha, \beta} E_{\alpha \bb} E_{\ba \beta, \gamma } \big|^2  =32 \sum_\gamma \big| \sum_{\alpha} E_{\alpha \ba, \gamma} E_{\ba \alpha}  \big|^2
\end{align*}
On the other hand, by the Codazzi equation for $E$, we find
\begin{align*}
|\nabla_b E|^2 = 4 \sum_{\alpha, \beta, \gamma} |E_{\alpha \bb, \gamma}|^2  & \geq  4 \ \sum_{\gamma} \bigg( |E_{\gamma \bg, \gamma}|^2 + \sum_{\alpha \neq \gamma} |E_{\alpha \ba, \gamma}|^2 + \sum_{\alpha \neq \gamma} |E_{\alpha \bg, \gamma}|^2 \bigg) \\
& = 4 \ \sum_{\gamma} \bigg( |E_{\gamma \bg, \gamma}|^2 + 2 \sum_{\alpha \neq \gamma} |E_{\alpha \ba, \gamma}|^2 \bigg)
\end{align*}
Thus to prove \eqref{c5}, we only need to demonstrate that for any $\gamma$,
\begin{align*}
\bigg| \sum_{\alpha} E_{\alpha \ba, \gamma} E_{\ba \alpha}  \bigg|^2 \leq  \frac{n}{n+1} \bigg( \sum_\alpha | E_{\alpha \ba} |^2 \bigg) \bigg( |E_{\gamma \bg, \gamma}|^2 + 2 \sum_{\alpha \neq \gamma} |E_{\alpha \ba, \gamma}|^2 \bigg)
\end{align*}
which is equivalent to
\begin{align} \label{c17}
\big| \sum_\alpha \lambda_\alpha \mu^\gamma_\alpha \big|^2 \leq \frac{n}{n+1} \bigg( \sum_\alpha | \lambda_\alpha|^2 \bigg) \bigg( |\mu^\gamma_\gamma|^2 +  2 \sum_{\alpha \neq \gamma} |\mu^\gamma_\alpha|^2 \bigg)
\end{align}
where $\mu^\gamma_\alpha = E_{\alpha \ba, \gamma}$.
But since $E$ is traceless, we know that $\sum_\alpha \mu^\gamma_\alpha =0$ and thus for any $\gamma$,
\begin{align*}
|\mu^\gamma_\gamma|^2 +  2 \sum_{\alpha \neq \gamma} |\mu^\gamma_\alpha|^2 =& |\mu^\gamma_\gamma|^2 +  \frac{n-1}{n} \sum_{\alpha \neq \gamma} |\mu^\gamma_\alpha|^2+ \frac{n+1}{n} \sum_{\alpha \neq \gamma} |\mu^\gamma_\alpha|^2 \\
\geq &|\mu^\gamma_\gamma|^2 +  \frac{1}{n} \big| \sum_{\alpha \neq \gamma} \mu^\gamma_\alpha \big|^2+ \frac{n+1}{n} \sum_{\alpha \neq \gamma} |\mu^\gamma_\alpha|^2 = \frac{n+1}{n} \sum_\alpha |\mu^\gamma_\alpha|^2
\end{align*}
which yields \eqref{c17} on account of Cauchy-Schwarz inequality. This completes the proof.
\end{proof}

The second term on the right side of \eqref{c4} can be estimated by Okumura's result:
\begin{lem}[\cite{okumura1974hypersurfaces}] \label{clm1}
Let $a_i, \ i=1, \dots, m$ be real numbers satisfying
\begin{align*}
\sum_{i=1}^m a_i =0, \quad \mbox{and } \  \sum_{i=1}^m a_i^2 = k^2
\end{align*}
Then we have
\begin{align*}
- \frac{m-2}{ \sqrt{m (m-1) } } k^3 \leq \sum_{i=1}^m a_i^3 \leq \frac{m-2}{ \sqrt{m (m-1) } } k^3
\end{align*}
\end{lem}
By this lemma, we obtain
\begin{align}
\bigg| \sum_{\alpha, \beta, \lambda} E_{\alpha \bl} E_{\lambda \bb} E_{\ba \beta} \bigg| \leq \frac{1}{2 \sqrt{2}} \frac{n-2}{\sqrt{n(n-1)}} |E|^3. \label{c6}
\end{align}
Now we use the method in \cite{huisken1985ricci} to estimate the third term on the right side of \eqref{c4}.

\begin{lem}
\begin{align}
\bigg| \sum_{\alpha, \beta, \lambda, \gamma}  E_{\gamma \bl } C_{ \bb \lambda \alpha \bg} E_{\ba \beta} \bigg|  \leq \frac{1}{4} \sqrt{\frac{2n^2 + 4n +3}{2 (n+1)(n+2)}} |E|^2  |C| . \label{c7}
\end{align}
\end{lem}

\begin{proof}
We denote
\begin{align*}
F =& (E_{\ba \beta} E_{\lambda \bm} + E_{\ba \lambda} E_{\beta \bm}) \theta^{\ba} \otimes \theta^{\beta} \otimes \theta^{\lambda} \otimes \theta^{\bm}  \\
&+ (E_{\alpha \bb} E_{\bl \mu} + E_{\alpha \bl} E_{\bb \mu}) \theta^{\alpha} \otimes \theta^{\bb} \otimes \theta^{\bl} \otimes \theta^{\mu}  \\
& - (E_{\bb \alpha} E_{\lambda \bm} + E_{\bb \lambda} E_{\alpha \bm}) \theta^{\alpha} \otimes \theta^{\bb} \otimes \theta^{\lambda} \otimes \theta^{\bm}  \\
&- (E_{\beta \ba} E_{\bl \mu} + E_{\beta \bl} E_{\ba \mu}) \theta^{\ba} \otimes \theta^{\beta} \otimes \theta^{\bl} \otimes \theta^{\mu}.
\end{align*}
It is easy to check that the tensor $F$ satisfies all algebraic properties of the Webster curvature, such as
\begin{align*}
F_{ \ba \beta \lambda \bm} = - F_{\beta \ba \lambda \bm} = - F_{ \ba \beta \bm \lambda}, \quad F_{ \ba \beta \lambda \bm} = F_{ \ba \lambda \beta \bm}.
\end{align*}
Therefore it can be decomposed into three orthogonal parts, i.e.  $F= T + P + Q$ where $T$,  $P$ and $Q$ are the traceless part, the ``partial trace" part and  the ``total
trace" part  of $F$ respectively.
More precisely,
\begin{align*}
P_{\ba \beta \lambda \bm} &= \frac{1}{n+2} \big( \tilde{F}_{\ba \beta} \delta_{\lambda \bm} + \tilde{F}_{\ba \lambda} \delta_{\beta \bm}  + \tilde{F}_{\lambda \bm} \delta_{\ba \beta} + \tilde{F}_{\beta \bm} \delta_{\lambda \ba}  \big) , \\
Q_{\ba \beta \lambda \bm}& = \frac{f}{n(n+1)} \big( \delta_{\ba \beta} \delta_{\lambda \bm} + \delta_{\ba \lambda} \delta_{\beta \bm} \big),
\end{align*}
where
\begin{align*}
\tilde{F}_{\lambda \bm } &= \tilde{F}_{\bm \lambda} = F_{\ba \alpha \lambda \bm} - \frac{f}{n} \delta_{\lambda \bm} = E_{\ba \lambda} E_{\alpha \bm} - \frac{f}{n} \delta_{\lambda \bm} ,\\
f &= F_{\ba \alpha \lambda \bl} = E_{\ba \lambda} E_{\alpha \bl} = \frac{1}{2} |E|^2.
\end{align*}
Note that $\sum_\alpha \tilde{F}_{\alpha \ba}=0$. Since $T, P, Q$ are mutually orthogonal, $|T|^2 = |F|^2 - |P|^2- |Q|^2$. But
\begin{align*}
\frac{1}{4} |F|^2 = F_{\ba \beta \lambda \bm} F_{\alpha \bb \bl \mu} =& (E_{\ba \beta} E_{\lambda \bm} + E_{\ba \lambda} E_{\beta \bm})  (E_{\alpha \bb} E_{\bl \mu} + E_{\alpha \bl} E_{\bb \mu}) \\
=& \frac{1}{2} |E|^4 + 2 Z 
\end{align*}
and
\begin{align*}
\frac{1}{4} |P|^2 = P_{\ba \beta \lambda \bm} P_{\alpha \bb \bl \mu} =& \frac{1}{(n+2)^2}  \big( \tilde{F}_{\ba \beta} \delta_{\lambda \bm} + \tilde{F}_{\ba \lambda} \delta_{\beta \bm} + \tilde{F}_{\lambda \bm} \delta_{\ba \beta} + \tilde{F}_{\beta \bm} \delta_{\lambda \ba} \big) \\
&  \quad  \times \big( \tilde{F}_{\alpha \bb} \delta_{\bl \mu} + \tilde{F}_{\alpha \bl} \delta_{\bb \mu} + \tilde{F}_{\bl \mu} \delta_{\alpha \bb} + \tilde{F}_{\bb \mu} \delta_{\bl \alpha} \big) \\
=& \frac{4}{n+2} \tilde{F}_{\bm \lambda} \tilde{F}_{\mu \bl} \\
=& \frac{4}{n+2} \big( E_{\ba \lambda} E_{\alpha \bm} - \frac{f}{n} \delta_{\lambda \bm} \big) \big( E_{\beta \bl} E_{\bb \mu} - \frac{f}{n} \delta_{\bl \mu} \big) \\
=& \frac{4}{n+2} (Z- \frac{1}{4n} |E|^4 )
\end{align*}
and
\begin{align*}
\frac{1}{4} |Q|^2 = Q_{\ba \beta \lambda \bm} Q_{\alpha \bb \bl \mu} =& \left( \frac{f}{n(n+1)} \right)^2 \big( \delta_{\ba \beta} \delta_{\lambda \bm} + \delta_{\ba \lambda} \delta_{\beta \bm} \big) \big( \delta_{\alpha \bb} \delta_{\bl \mu} + \delta_{\alpha \bl} \delta_{\bb \mu} \big)  \\
=& \frac{1}{2n (n+1) } |E|^4
\end{align*}
where $Z = E_{\ba \beta} E_{\lambda \bm} E_{\alpha \bl} E_{\bb \mu}= E_{\ba \beta} E_{\bb \mu} E_{\bm \lambda} E_{\bl \alpha}$. Thus
\begin{align*}
\frac{1}{4} |T|^2 =&  \frac{1}{4} |F|^2 -\frac{1}{4} |P|^2- \frac{1}{4} |Q|^2 = \frac{n^2 +3 n +3}{ 2 (n+1) (n+2)} |E|^4 + \frac{2n}{n+2} Z \\
\leq & \frac{n^2 +3 n +3}{ 2 (n+1) (n+2)} |E|^4 + \frac{n}{2(n+2)} |E|^4 = \frac{2n^2 + 4n +3}{2 (n+1)(n+2)} |E|^4 .
\end{align*}
Since $ C_{ \bb \lambda \alpha \bm} = C_{\bm \lambda \alpha \bb} $, we can complete the proof by
\begin{align*}
\big| \sum_{\alpha, \beta, \lambda, \gamma} E_{\mu \bl } C_{ \bb \lambda \alpha \bm} E_{\ba \beta}  \big| =& \frac{1}{8} \la F, C \ra = \frac{1}{8} \la T, C \ra \leq \frac{1}{8} | T| |C| \\
\leq & \frac{1}{4} \sqrt{\frac{2n^2 + 4n +3}{2 (n+1)(n+2)}} |E|^2 |C| .
\end{align*}
\end{proof}

Substituting \eqref{c5} \eqref{c6} and \eqref{c7} into \eqref{c4}, we have
\begin{lem}
If $(M, HM, J_b, \theta)$ is a Sasakian manifold with constant pseudo-Hermitian scalar curvature and the dimension $2n +1 \geq 5$, the traceless pseudo-Hermitian Ricci tensor satisfies
\begin{align}
\frac{1}{2} \triangle_b |E|^2 + \bigg( b (x)  - \frac{2 \rho}{n+1} \bigg) |E|^2 -  |\nabla_b E|^2 \geq 0, \quad \mbox{on } M \label{c18}
\end{align}
and
\begin{align}
|E| \triangle_b |E| + \bigg( b (x)  - \frac{2 \rho}{n+1} \bigg) |E|^2 - \frac{1}{n} |\nabla_b |E||^2 \geq 0, \quad \mbox{weakly on } \  M \label{c9}
\end{align}
where $b (x) = \sqrt{\frac{ 2 n^2 + 4n +3}{ (n+1)(n+2)}}( |E| + \frac{1}{\sqrt{2}} |C|) $.
\end{lem}

Now we apply Corollary \ref{blm1} to \eqref{c9} and get the following theorem

\begin{thm} \label{ctm3}
Let $(M, HM, J_b, \theta)$ be a complete noncompact Sasakian manifold with zero pseudo-Hermitian scalar curvature, positive CR Yamabe constant and the dimension $2n+1 \geq 5$. Assume that
\begin{align}
\int_{B_r} |E|^\sigma \theta \wedge (d \theta)^n = o (r^2), \quad \mbox{as} \ r \rightarrow \infty \label{c16}
\end{align}
and
\begin{align}
\frac{1}{\sqrt{2}} || C ||_{L^{n+1} (M)} + ||E||_{L^{n+1} (M)} <  \frac{ 2n \sigma -2 n +2}{ \sigma^2  \sqrt{n+1}} \sqrt{\frac{ n+2}{ 2 n^2 + 4n +3}}  \lambda(M) \label{c15}
\end{align}
where $\sigma \geq 2$.
Then $(M, HM, J_b, \theta)$ is a pseudo-Einstein manifold. Actually it is pseudo-Hermitian Ricci-flat.
\end{thm}

\begin{proof}
Since the pseudo-Hermitian scalar curvature is zero, by \eqref{c9}, $|E|$ satisfies
\begin{align*}
|E| \triangle_b |E| +  b (x)  |E|^2 - \frac{1}{n} |\nabla_b |E||^2 \geq 0, \quad \mbox{weakly on } \  M.
\end{align*}
On the other hand, the assumption \eqref{c15} yields
\begin{align*}
|| b(x) ||_{L^{n+1} (M)} < \frac{2n \sigma - 2n +2}{ \sigma^2 (n+1)} \lambda (M).
\end{align*}
It guarantees the existence of sufficiently small $ \epsilon $ and $ \delta $ such that $\delta >0$, $0 < \epsilon < \sigma + \frac{1}{n}-1$ and
\begin{align*}
||   b (x) ||_{L^{n+1 } (M)} < \tilde{C}_{\delta, \epsilon} \lambda (M)
\end{align*}
where $\tilde{C}_{\delta, \epsilon} = \frac{2n}{n+1} \frac{\sigma + \frac{1}{n} -1 -\epsilon }{\sigma^2 (1+ \delta)}  $. Hence using Corollary \ref{blm1} with
\begin{align*}
\psi = |E|, \quad a(x) = b(x), \quad \mbox{and } \ B= - \frac{1}{n},
\end{align*}
we conclude that $|E| = 0$.
\end{proof}

The most interesting case is $\sigma = n+1$ since the $L^{n+1}$ norm of the Chern-Moser tensor is a CR conformal invariant and in this case,  the conditions \eqref{c16}, \eqref{c15} merge into one.

\begin{thm} \label{cco1}
Let $(M, HM, J_b, \theta)$ be a complete noncompact Sasakian manifold with zero pseudo-Hermitian scalar curvature, positive CR Yamabe constant and the dimension $2n+1 \geq 5$. Assume that
\begin{align*}
\frac{1}{\sqrt{2}} || C ||_{L^{n+1} (M)} + ||E||_{L^{n+1} (M)} < \frac{2n^2 + 2}{(n+1)^{\frac{5}{2} } } \sqrt{ \frac{n+2}{2n^2+4n +3}} \lambda(M).  
\end{align*}
Then $(M, HM, J_b, \theta)$ is a pseudo-Hermitian Ricci-flat manifold.
\end{thm}

If the pseudo-Hermitian scalar curvature is negative, the result and the corresponding proof are similar.

\begin{thm}
Let $(M, HM, J_b, \theta)$ be a complete noncompact Sasakian manifold with constant negative pseudo-Hermitian scalar curvature, positive CR Yamabe constant and the dimension $2n+1 \geq 5$. Assume that
\begin{align}
\int_{B_r} |E|^\sigma \theta \wedge (d \theta)^n = o (r^2), \quad \mbox{as} \ r \rightarrow \infty
\end{align}
and
\begin{align}
\frac{1}{\sqrt{2}} || C ||_{L^{n+1} (M)} + ||E||_{L^{n+1} (M)} < \frac{ 2n \sigma -2 n +2}{ \sigma^2  \sqrt{n+1}} \sqrt{\frac{ n+2}{ 2 n^2 + 4n +3}} \lambda(M)
\end{align}
where $ 2 \leq \sigma < n-1$.
Then $(M, HM, J_b, \theta)$ is a pseudo-Einstein manifold.
\end{thm}

Before introducing the case of positive pseudo-Hermitian scalar curvature, we recall a variation of Myers' theorem due to I. Hasegawa and M. Seino \cite{hasegawa1981}. Since the proof is simple, we provided it here for completeness.

\begin{lem} \label{clm2}
Let $(M, HM, J_b, \theta)$ be a Sasakian manifold with positive pseudo-Hermitian Ricci curvature, that is $R_{\alpha \bb} \geq c \delta_{\alpha \bb}$ for some positive constant $c$. Then $M$ is compact with finite fundamental group.
\end{lem}
\begin{proof}
We consider the D-homothetic transformation $\tilde{\theta} = \lambda \theta$ for some positive constant $\lambda$. At this time, $(M, HM, J_b, \tilde{\theta})$ is also a Sasakian manifold. Moreover Lemma \ref{alm1} yields that its pseudo-Hermitian Ricci curvature $\tilde{R}_{\alpha \bb} = \lambda^{-1} R_{\alpha \bb}$. Hence by \eqref{a10}, its Riemannian Ricci curvature $\widetilde{Ric}^\theta$ is
\begin{gather*}
\widetilde{Ric}^\theta_{\alpha \bb} = \lambda^{-1} R_{\alpha \bb} - 2 \delta_{\alpha \bb} \geq (c \lambda^{-1} -2) \delta_{\alpha \bb}, \\
\widetilde{Ric}^\theta_{\alpha \beta} = \widetilde{Ric}^\theta_{\alpha 0}=0, \quad \mbox{and } \  \widetilde{Ric}^\theta_{00} = 2n.
\end{gather*}
By choosing sufficiently small $\lambda $, the Riemannian Ricci curvature will be positive definite. Thus this lemma follows from Myers' Theorem.
\end{proof}


\begin{thm} \label{ctm1}
Let $(M, HM, J_b, \theta)$ be a complete Sasakian manifold with constant positive pseudo-Hermitian scalar curvature, positive CR Yamabe constant and the dimension $2n+1 \geq 5$. Assume that for any $\sigma \geq 2$,
\begin{align}
\int_{B_r} |E|^\sigma \theta \wedge (d \theta)^n = o (r^2), \quad \mbox{as} \ r \rightarrow \infty, \label{c8}
\end{align}
and
\begin{align}
 \frac{1}{\sqrt{2}} || C ||_{L^{n+1} (M)} + ||E||_{L^{n+1} (M)}  <  C_{n \sigma} \sqrt{\frac{ (n+1)(n+2)}{ 2 n^2 + 4n +3}} \lambda(M) \label{c10}
\end{align}
where
\begin{numcases}{C_{n \sigma} = }
\frac{ 2n \sigma -2 n +2}{(n+1) \sigma^2} , & for $n=2 \mbox{ or } 3, \ \sigma \geq 2 $, \nonumber \\
\frac{2}{n+1}, & for $n \geq 4, \ 2 \leq \sigma < n-1$,  \nonumber \\
\frac{ 2n \sigma -2 n +2}{(n+1) \sigma^2} , & for $n \geq 4, \ \sigma \geq n-1$. \nonumber
\end{numcases}
Then $M$ is compact.
\end{thm}

\begin{proof}
We prove it by contradiction. Suppose that $M$ is noncompact.
Let $\kappa$ be a sufficiently small positive number such that
\begin{align}
||b(x)||_{ L^{n+1} (M) } \leq& \sqrt{\frac{ 2 n^2 + 4n +3}{ (n+1)(n+2)}} \left( \frac{1}{\sqrt{2}} || C ||_{L^{n+1} (M)} + ||E||_{L^{n+1} (M)} \right) \nonumber \\
 < & \  ( C_{n \sigma}- \kappa)  \lambda(M) . \label{c11}
\end{align}
We choose some proper $0 < \epsilon < \sigma + \frac{1}{n} - 1 $ and $\delta > 0$ such that the coefficient $ \tilde{C}_{\delta, \epsilon} = C_{n \sigma}- \kappa  $ in \eqref{f9}. Moreover, the inequality
$
-\frac{2\rho}{n+1} + \tilde{C}_{\delta, \epsilon} \rho  \leq 0
$ is always true. Hence on account of  \eqref{c9} and the initial assumption that $M$ is noncompact, Corollary \ref{blm1} yields that $E \equiv 0$ and then $R_{\lambda \bm} = \frac{\rho}{n}$. But since $\rho$ is positive, Lemma \ref{clm2} guarantees that $M$ is compact which contradicts with our initial assumption. Thus $M$ must be compact.
\end{proof}

The method in \cite{hebey1996effectivel} enables us to obtain the pseudo-Einstein property in the interesting case $\sigma = n+1$.

\begin{cor} \label{cco2}
Let $(M, HM, J_b, \theta)$ be a complete Sasakian manifold with constant positive pseudo-Hermitian scalar curvature, positive CR Yamabe constant and the dimension $2n+1 \geq 5$. Assume that
\begin{align}
\frac{1}{\sqrt{2}} || C ||_{L^{n+1} (M)} + ||E||_{L^{n+1} (M)}  <  \frac{2n^2 + 2}{(n+1)^{\frac{5}{2} } } \sqrt{ \frac{n+2}{2n^2+4n +3}} \lambda(M). \label{c14}
\end{align}
Then $(M, HM, J_b, \theta)$ is  a compact pseudo-Einstein manifold and the real first Chern class of the horizontal bundle $HM$ vanishes.
\end{cor}

\begin{proof}
Since \eqref{c14} implies \eqref{c8} and \eqref{c10} with $\sigma =n+1$, $M$ is compact by Theorem \ref{ctm1}. Hence it suffices to prove that $(M, HM, J_b, \theta)$ is pseudo-Einstein.
Integrating \eqref{c18} over $M$ and using the H\"older inequality, we have
\begin{align}
||b||_{L^{n+1}} ||E||^2_{L^{p}} - \frac{2 \rho}{n+1}  ||E||^2_{L^2}  - \frac{n+1}{n} ||\nabla_b |E| ||^2_{L^2} \geq 0. \label{c13}
\end{align}
Since the CR Yamabe constant is positive, we can estimate $||E||^2_{L^{p}}$ by
\begin{align*}
||E||^2_{L^{p}} \leq \lambda(M)^{-1} ( b_n ||\nabla_b |E| ||^2_{L^2} + \rho ||E||^2_{L^2}  )
\end{align*}
where $p= b_n = 2 + \frac{2}{n}$.
Substituting it to \eqref{c13}, we find
\begin{align*}
\bigg( \frac{ ||b(x)||_{L^{n+1}} }{ \lambda (M) } - \frac{1}{2} \bigg) b_n ||\nabla_b |E| ||^2_{L^2} + \bigg( \frac{ ||b(x)||_{L^{n+1}} }{ \lambda (M)} - \frac{2}{n+1} \bigg) \rho ||E||^2_{L^2} \geq 0 .
\end{align*}
Hence $E=0$ on account of \eqref{c14} and then $(M, HM, J_b, \theta)$ is pseudo-Einstein. From \cite{lee1988psuedo}, we know that the real first Chern class of the horizontal bundle $HM$ vanishes. This completes the proof.
\end{proof}

By Omori-Yau maximum principle (cf.  \cite{yau1975harmonic}), we have the following $L^{\infty}$ pinching theorem.

\begin{thm} \label{ctm2}
Let $(M, HM, J_b, \theta)$ be a complete Sasakian manifold of dimension $2n+1 \geq 5$ and with constant positive pseudo-Hermitian scalar curvature $\rho $. Assume that
\begin{align}
\sup_M \left( \sqrt{2} |E| +  |C| \right)  < \sqrt{ \frac{8 (n+2 )}{ (n+1) (2n^2+4n +3) } } \rho . \label{c12}
\end{align}
Then $(M, HM, J_b, \theta)$ is a pseudo-Einstein manifold. In particular, $M $ is compact and the real first Chern class of the horizontal bundle $HM$ vanishes.
\end{thm}

\begin{proof}
We prove it by contradiction. Assume that $(M, HM, J_b, \theta)$ is not pseudo-Einstein. Then $\sup_M |E|^2$ is nonzero.
The condition \eqref{c12} implies that the pseudo-Hermitian Ricci curvature is bounded. So is the Riemannian Ricci curvature $Ric^\theta$ due to the relations \eqref{a10}.
Hence if $M$ is noncompact, we can apply Omori-Yau maximum principle to $|E|^2$; if $M $ is compact, we use the standard maximum principle. In either case, we can get a sequence $\{ x_k \}$ such that
\begin{align*}
\lim_{k \rightarrow + \infty } |E|^2 (x_k)= \sup_M |E|^2, \  \lim_{k \rightarrow + \infty } | \nabla E | =0, \mbox{ and } \limsup_{k \rightarrow + \infty } \triangle |E|^2  (x_k) \leq 0.
\end{align*}
By \eqref{c18} and $E_{\alpha \bb, 0}=0$, we have
\begin{align*}
\frac{1}{2} \triangle |E|^2 \geq |\nabla E|^2 +\bigg( \frac{2 \rho}{n+1} - b (x)  \bigg) |E|^2  \geq \bigg( \frac{2 \rho}{n+1} - b (x)  \bigg) |E|^2 .
\end{align*}
Considering it at $x_k$ and taking the supremum limit as $k \rightarrow + \infty$, we discover
\begin{align*}
\limsup_{k \rightarrow + \infty} \bigg( \frac{2 \rho}{n+1} - b (x)  \bigg) (x_k) \leq 0
\end{align*}
which leads a contradiction with \eqref{c12}. We have proved that $ (M, HM, J_b, \theta) $ is pseudo-Einstein. The compactness of $M$ is due to Lemma \ref{clm2}. Hence by \cite{lee1988psuedo}, the real first Chern class of the horizontal bundle $HM$ vanishes. This completes the proof.
\end{proof}

It is known that if a compact K\"ahler manifold has quasi-positive orthogonal
bisectional curvature and constant scalar curvature, then the K\"ahler metric
is K\"ahler-Einstein (cf. \cite{guan2009uniqueness,howard1981compact}). One may prove this result by using the Ricci identity and the Codazzi
equation for the Ricci tensor. For a pseudo-Hermitian manifold, especially a Sasakian manifold, we have the Codazzi type equation \eqref{a1}
and the Ricci identity \eqref{b1} too. These properties for the pseudo-Hermitian Ricci tensor enable us to derive a similar theorem.

\begin{dfn}
A strictly pseudoconvex CR manifold $(M, HM, J_b, \theta)$ has nonnegative orthogonal pseudo-Hermitian sectional curvature if for any point $x \in M$ and any $ X, Y \in T_{1,0} M_x$ with $g_\theta (X, \overline{Y}) =0$, we have
\begin{align*}
g_\theta \big( R(X, \overline{X}) Y, \overline{Y} \big) \geq 0.
\end{align*}
If in addition there exists some $x_0 \in M$  such that the above inequality is strict for any $ X, Y \in T_{1,0} M_{x_0}$, then $(M, HM, J_b, \theta)$ is said to have quasi-positive orthogonal pseudo-Hermitian sectional curvature.
\end{dfn}

\begin{thm} \label{btm1}
Let $(M, HM, J_b, \theta)$ be a compact $(2n+1)$-Sasakian manifold with constant pseudo-Hermitian scalar curvature. Assume that it has quasi-positive orthogonal pseudo-Hermitian sectional curvature,  then $(M, HM, J_b, \theta)$ is a pseudo-Einstein manifold. In particular, the real first Chern class of the horizontal bundle $HM$ vanishes.
\end{thm}

\begin{proof}
Since the pseudo-Hermitian torsion vanishes, we have $R_{\lambda \bm, 0 } =0 $ and $R_{\alpha \bb, \lambda} = R_{\lambda \bb, \alpha}$ by \eqref{a1} and \eqref{a2}.
The constancy of the pseudo-Hermitian scalar curvature and the Ricci identity \eqref{b1} yield
\begin{align*}
\triangle_b |R_{\alpha \bb}|^2 =& (R_{\alpha \bb} R_{\ba \beta})_{, \gamma \bg}+ (R_{\alpha \bb} R_{\ba \beta})_{, \bg \gamma} \\
=& 4 R_{\alpha \bb, \gamma} R_{\ba \beta, \bg} + 2 R_{\alpha \bb, \gamma \bg} R_{\ba \beta} + 2 R_{\alpha \bb, \bg \gamma} R_{\ba \beta} \\
=& 4 R_{\alpha \bb, \gamma} R_{\ba \beta, \bg} + 4 R_{\mu \bb} R_{\beta \ba } R_{\alpha \bm} - 4 R_{\bb \mu \alpha \bg} R_{\gamma \bm} R_{\ba \beta} .
\end{align*}
Since $(R_{\alpha \bb})$ is a Hermitian matrix, we can choose some proper basis $\{ \theta^\alpha \}$ such that $(R_{\alpha \bb})$ is diagonal at a given point and assume that $\{ \lambda_\alpha \}_{\alpha=1}^n$ are the eigenvalues. Then
\begin{align}
\triangle_b |R_{\alpha \bb}|^2 =&  4 R_{\alpha \bb, \gamma} R_{\ba \beta, \bg} + 4 \sum_\alpha \lambda_\alpha^3 - 4 \sum_{\alpha, \: \beta} R_{\ba \alpha \beta \bb} \lambda_\beta \lambda_\alpha  \nonumber \\
=&4 R_{\alpha \bb, \gamma} R_{\ba \beta, \bg} + 2\sum_{\alpha \neq \beta} R_{\ba \alpha \beta \bb} (\lambda_\alpha - \lambda_\beta )^2 \geq 0 , \label{b2}
\end{align}
due to $R_{\ba \alpha \beta \bb} \geq 0$ by assumption. Bony's maximum principle (Lemma \ref{alm3}) guarantees that $|R_{\alpha \bb}|$ is constant. Hence $R_{\alpha \bb, \gamma}=0$ and $\sum_{\alpha \neq \beta} R_{\ba \alpha \beta \bb} (\lambda_\alpha - \lambda_\beta ) =0$. The first equation says that the pseudo-Hermitian Ricci curvature is parallel and thus the eigenvalues are constant numbers. But since $R_{\ba \alpha \beta \bb} >0$ at some point, $\lambda_\alpha = \lambda_\beta$ at this point and therefore they are equal. So $(M, HM, J_b, \theta)$ is a pseudo-Einstein manifold. Hence by \cite{lee1988psuedo}, the real first Chern class of the horizontal bundle $HM$ vanishes. This completes the proof.
\end{proof}

\begin{rmk}
This theorem was first generalized to Sasakian case by X. Zhang in \cite{zhang2009note}. We recapture it in a simpler way.
Moreover all of the results in this section can be generalized to complete strictly pseudoconvex CR manifolds with divergence free  pseudo-Hermitian torsion, that is $A_{\alpha \beta, \bb}$ =0, since both the Ricci formula of pseudo-Hermitian Ricci curvature with $R_{\lambda \bm, 0} =0 $ and the Codazzi equation still hold in this case.
\end{rmk}

\section{Rigidity Theorems of Sasakian space forms} \label{s2}
In this section, we establish some rigidity theorems on complete Sasakian pseudo-Einstein manifolds.
Let $(M, HM , J_b, \theta)$ be a Sasakian pseudo-Einstein manifold with pseudo-Hermitian scalar curvature. Under the assumption, by \eqref{c3},  the Chern-Moser tensor becomes
\begin{align}
C_{\ba \beta \lambda \bm} = R_{\ba \beta \lambda \bm} - \frac{\rho}{n(n+1)} (\delta_{\ba \beta} \delta_{\lambda \bm} + \delta_{\ba \lambda} \delta_{\beta \bm}) . \label{d1}
\end{align}
By Lemma \ref{alm2}, we know that the pseudo-Hermitian scalar curvature is constant. Hence \eqref{d1} yields that
the Chern-Moser tensor satisfies the Bianchi-type identity, i.e.
$
C_{\ba \beta \lambda \bm, \gamma} = C_{\ba \beta \gamma \bm, \lambda}.
$
Similarly, the equation \eqref{a5} gives
$
C_{\ba \beta \lambda \bm, 0} =0.
$
From the Ricci identity \eqref{d10} and \eqref{d1}, we derive
\begin{align*}
C_{\ba \beta \gamma \bm, \lambda \bg} -& C_{\ba \beta \gamma \bm, \bg \lambda} = -C_{\bn \beta \gamma \bm} C_{\ba \nu \lambda \bg} + C_{\ba \nu \gamma \bm} C_{\bn \beta \lambda \bg}- C_{\ba \beta \gamma \bn } C_{\bm \nu \lambda \bg} + \frac{\rho}{n} C_{\ba \beta \lambda \bm} .
\end{align*}
Under these preparations, the sub-Laplacian of $ |C_{\ba \beta \lambda \bm}|^2 $ is given by
\begin{align}
& \frac{1}{2} \triangle_b |C_{\ba \beta \lambda \bm}|^2 = 2 C_{\ba \beta \lambda \bm, \gamma} C_{\alpha \bb \bl \mu, \bg} + C_{\ba \beta \lambda \bm, \gamma \bg} C_{\alpha \bb \bl \mu} + \overline{ C_{\ba \beta \lambda \bm, \gamma \bg} C_{\alpha \bb \bl \mu} } \nonumber \\
&= 2 | C_{\ba \beta \lambda \bm, \gamma} |^2 +  C_{\ba \beta\gamma \bm,\lambda \bg} C_{\alpha \bb \bl \mu} + \overline{ C_{\ba \beta \gamma \bm,\lambda \bg} C_{\alpha \bb \bl \mu} } \nonumber \\
&= 2 | C_{\ba \beta \lambda \bm, \gamma} |^2 -4 C_{\bl \alpha \mu \bb} C_{\bm \beta \gamma \bn} C_{\bg \nu \lambda \ba} + 2 C_{\bb \alpha \mu \bl} C_{\ba \nu \gamma \bm} C_{\bn \beta \lambda \bg} + \frac{2 \rho}{n} |C_{\ba \beta \lambda \bm}|^2 . \label{d2}
\end{align}
To estimate the second term of the last line, we consider the $n^2 \times n^2 $ Hermitian matrix $(D_{\overline{(\lambda \ba)} (\mu \bb)} )$ with its entry $D_{\overline{(\lambda \ba)} (\mu \bb)} = C_{\bl \alpha \mu \bb}$. Since the Chern-Moser tensor is traceless, we can use Lemma \ref{clm1} to deduce that
\begin{gather}
\bigg| \sum_{\alpha, \beta, \lambda, \mu, \nu, \gamma } C_{\bl \alpha \mu \bb} C_{\bm \beta \gamma \bn} C_{\bg \nu \lambda \ba} \bigg| \leq \frac{n^2 -2}{ \sqrt{n^2 (n^2-1 ) }} \bigg| \sum_{\alpha, \beta, \lambda, \mu} |C_{\ba \beta \lambda \bm}|^2   \bigg|^{\frac{3}{2}}  . \label{d3}
\end{gather}
For the third term, by considering the Hermitian matrix $(H_{\overline{(\beta \lambda)} (\alpha \mu)})$ with its entry $H_{\overline{(\beta \lambda)} (\alpha \mu)} = C_{\bb \alpha \mu \bl}$, we have a similar estimate. By definition, we have
\begin{align*}
|C|^2 =4 \sum_{\alpha, \beta, \lambda, \mu} |C_{\ba \beta \lambda \bm}|^2 , \quad
|\nabla_b C|^2 = 8 \sum_{\alpha, \beta, \lambda, \mu, \gamma} |C_{\ba \beta \lambda \bm, \gamma}|^2
\end{align*}
where $\nabla_b C$ is the horizontal part of $\nabla C$.
Then \eqref{d2} becomes
\begin{align} \label{d11}
\frac{1}{2} \triangle_b |C|^2 \geq  | \nabla_b C|^2 - \frac{3 ( n^2 -2) }{ \sqrt{n^2 (n^2-1 ) }} |C|^3 + \frac{2 \rho}{n} |C|^2.
\end{align}

To deal with the first term on the right side of \eqref{d11}, we need the following type of Kato inequality
\begin{lem}[\cite{bando1989construction}] \label{dlm1}
Suppose $S, T$ are tensors having the same symmetry as the curvature $R$, and the covariant derivative $\nabla R$ of the curvature tensor of the Einstein metric $g$ respectively. Then there exists $\delta= \delta (m) $ such that
\begin{align*}
(1+ \delta ) \big| (S, T) \big|^2 \leq |S|^2 |T|^2,
\end{align*}
where $(S, T)$ is a 1-form defined by $(S, T) (X) = \big( S, T(X) \big)$ for a tangent vector $X$. Moreover if $g$ is K\"ahler, we can take $\delta= \frac{4}{ m+2}$. If $m=4$, and $g$ is self-dual or anti self-dual, we can take $\delta= \frac{2}{3}$.
\end{lem}

The proof of Lemma \ref{dlm1} only involves the algebraic symmetric properties of $S$ and $ T$, such as first Bianchi identity, second Bianchi identity and antisymmetric when $g$ is K\"ahler. In our case, the tensors $C$ and $ \nabla_b C$ have these properties. Hence we can repeat the proof $(m= 2n)$ and get
\begin{align}
\frac{n+3}{n+1} |\la C, \nabla_b C \ra|^2 \leq |C|^2 |\nabla_b C|^2. \label{d7}
\end{align}

\begin{lem}
Let $(M, HM , J_b, \theta)$ be a Sasakian pseudo-Einstein manifold with pseudo-Hermitian scalar curvature. Then the Chern-Moser tensor satisfies
\begin{align}
\frac{1}{2} \triangle_b |C|^2 \geq  | \nabla_b C|^2 - \varepsilon |C|^3 + \frac{2 \rho}{ n} |C|^2 , \quad \mbox{on } \ M \label{d6}
\end{align}
and
\begin{align}
|C| \triangle_b |C| + \bigg( d(x) - \frac{2 \rho}{n} \bigg) |C|^2 - \frac{2}{n+1} |\nabla_b |C||^2 \geq 0, \quad \mbox{weakly on } \ M \label{d4}
\end{align}
where $\varepsilon=  \frac{3 ( n^2 -2) }{ \sqrt{n^2 (n^2-1 ) }}$ and $d(x) = \frac{3 ( n^2 -2) }{ \sqrt{n^2 (n^2-1 ) }} |C|$.
\end{lem}

When the pseudo-Hermitian scalar curvature is zero, we can apply Corollary \ref{blm1} to \eqref{d4} and repeat the proof of Theorem \ref{ctm3} to obtain some rigidity theorems for the Chern-Moser tensor. On account of Lemma \ref{alm4}, we have

\begin{thm} \label{dtm2}
Let $(M, HM , J_b, \theta)$ be a noncompact Sasakian pseudo-Einstein manifold with zero pseudo-Hermitian scalar curvature, positive CR Yamabe constant and the dimension $2n +1 \geq 5$. Assume that for some $\sigma \geq 2$,
\begin{align}
\int_{B_r} |C|^\sigma \theta \wedge (d \theta)^n = o (r^2), \quad  \mbox{as  } \ r \rightarrow + \infty,
\end{align}
and
\begin{align}
|| C ||_{L^{n+1} (M) } <  \frac{2 n^2}{ 3(n^2-2)} \sqrt{ \frac{n-1}{n+1} } \bigg( \sigma + \frac{2}{n+1}-1 \bigg) \sigma^{-2} \lambda (M) . \label{d5}
\end{align}
Then the Chern-Moser tensor vanishes and $(M, HM , J_b, \theta)$ has zero pseudo-Hermitian sectional curvature. If in addition $M$ is simply connected, then $(M, HM , J_b, \theta)$ is D-homothetic to Heisenberg group.
\end{thm}

Combining Theorem \ref{dtm2} with Theorem \ref{cco1}, we have the following rigidity theorem which characterizes Heisenberg group.
\begin{thm}
Let $(M, HM, J_b, \theta)$ be a simply connected complete noncompact Sasakian manifold with zero pseudo-Hermitian scalar curvature, positive CR Yamabe constant and the dimension $2n+1 \geq 5$. Assume that
\begin{align*}
 || C ||_{L^{n+1} (M)} +\sqrt{2} ||E||_{L^{n+1} (M)} < \frac{2n^2(n^2+n+2)}{3(n+1)^3(n^2-2)} \sqrt{ \frac{n-1}{n+1} } \lambda (M).
\end{align*}
Then $(M, HM , J_b, \theta)$ is D-homothetic to Heisenberg group.
\end{thm}

For the negative pseudo-Hermitian scalar curvature, we have a similar result.

\begin{thm}
Let $(M, HM , J_b, \theta)$ be a complete noncompact Sasakian pseudo-Einstein manifold with negative pseudo-Hermitian scalar curvature, positive CR Yamabe constant and the dimension $2n +1 \geq 9$. Assume that
\begin{align*}
\int_{B_r} |C|^\sigma \theta \wedge (d \theta)^n = o (r^2), \quad  \mbox{as  } \ r \rightarrow + \infty,
\end{align*}
and
\begin{align*}
|| C ||_{L^{n+1} (M) } <  \frac{2 n^2}{ 3(n^2-2)} \sqrt{ \frac{n-1}{n+1} } \bigg( \sigma + \frac{2}{n+1}-1 \bigg) \sigma^{-2} \lambda (M)
\end{align*}
for some $2 \leq \sigma < \frac{1}{2(n+1)} \left( n^2 + \sqrt{n^4 - 4 n^3 + 4 n^2} \right)$.
Then the Chern-Moser tensor vanishes and $(M, HM , J_b, \theta)$ has constant negative pseudo-Hermitian sectional curvature.
\end{thm}

For Riemannian compact Einstein manifolds, the authors in \cite{hebey1996effectivel,shen1990some} established some $L^p$ rigidity theorems to characterize  the spheres.  For compact Sasakian pseudo-Einstein manifolds, M. Itoh and H. Satoh \cite{itoh2004isolation} gave an  $L^{n+\frac{1}{2}}$ gap condition for the Chern-Moser tensor to characterize $S^{2n+1}$. Note that their gap condition is not CR conformal invariant. Let's therefore attempt to find an $L^{n+1}$ rigidity theorem for compact Sasakian pseudo-Einstein manifolds.

\begin{thm} \label{dtm1}
Let $(M, HM , J_b, \theta)$ be a compact Sasakian pseudo-Einstein manifold with positive pseudo-Hermitian scalar curvature, positive CR Yamabe constant and the dimension $2n +1 \geq 5$. Assume that
\begin{align}
||C ||_{L^{n+1} (M)} < C_1  \lambda(M), \label{d9}
\end{align}
where
\begin{numcases}{C_1 =}
\frac{5}{9 \sqrt{3}}, & for $n=2 $, \nonumber \\
\frac{9  \sqrt{2}}{56}, & for $n=3$, \nonumber \\
\frac{2 \sqrt{n^2 -1}}{3( n^2-2 )}. & for $n \geq 4$. \nonumber
\end{numcases}
Then $C \equiv 0$ which means that $(M, HM , J_b, \theta)$ has constant positive pseudo-Hermitian sectional curvature. Furthermore, if  $M$ is simply connected, then $(M, HM , J_b, \theta)$ is D-homothetic to $S^{2n+1}$.
\end{thm}

\begin{proof}
By integrating \eqref{d6} over $M$ and using \eqref{d7} and the divergence theorem, we get
\begin{align}
\frac{n+3}{n+1} || \nabla_b |C| ||_{L^2 }^2 + \frac{2 \rho}{n} ||C||_{L^2}^2 \leq \varepsilon ||C||^3_{L^3}. \label{d8}
\end{align}
Moreover, by the H\"older inequality and the assumption of positive CR Yamabe constant,
\begin{align*}
||C||^3_{L^3 } \leq ||C||_{L^{n+1} } ||C||_{L^{\frac{2n+2}{n}}}^2 \leq \frac{ ||C||_{L^{n+1} } }{\lambda (M)} \bigg( \frac{2n+2}{n} || \nabla_b |C| ||_{L^2 }^2 + \rho || C ||_{L^2 }^2  \bigg).
\end{align*}
Substituting the above inequality into \eqref{d8}, we have
\begin{align*}
\big( \frac{2 \varepsilon (n+1)}{n \lambda(M)} || C ||_{L^{n+1}} - \frac{n+3}{n+1} \big) || \nabla_b |C| ||_{L^2}^2 +  \big( \frac{\varepsilon }{\lambda (M)} || C ||_{L^{n+1}} -\frac{2}{n}  \big)\rho || C ||_{L^2}^2 \geq 0.
\end{align*}
Combining with \eqref{d9}, we conclude that $C \equiv 0$ and then $(M, HM , J_b, \theta)$ has constant positive pseudo-Hermitian sectional curvature. The rest part of this theorem is due to Lemma \ref{alm4}.
\end{proof}

Combining Theorem \ref{dtm1} with Corollary \ref{cco2}, we have

\begin{cor}
Let $(M, HM, J_b, \theta)$ be a simply-connected complete Sasakian manifold with constant positive pseudo-Hermitian scalar curvature, positive CR Yamabe constant and the dimension $2n+1 \geq 5$. Assume that
\begin{align*}
|| C ||_{L^{n+1} (M)} + \sqrt{2} \:  ||E||_{L^{n+1} (M)}  < C_1  \lambda(M),
\end{align*}
where $C_1$ is defined as Theorem \ref{dtm1}.
Then $(M, HM , J_b, \theta)$ is D-homothetic to $S^{2n+1}$.
\end{cor}

A direct consequence of \eqref{d6} and the maximum principle is the following $L^\infty$ pinching theorem.

\begin{thm} \label{dtm3}
Let $(M, HM , J_b, \theta)$ be a compact Sasakian pseudo-Einstein manifold of dimension $2n +1 \geq 5$ and with positive pseudo-Hermitian scalar curvature. Assume that
\begin{align*}
\sup_M |C| < \frac{2\sqrt{n^2 -1}}{3(n^2-2)} \: \rho .
\end{align*}
Then $C \equiv 0$ and $(M, HM , J_b, \theta)$ has constant positive pseudo-Hermitian sectional curvature. Furthermore, if $M$ is simply connected, then $(M, HM , J_b, \theta)$ is D-homothetic to $S^{2n+1}$.
\end{thm}

A direct calculation shows that
$$
\frac{2\sqrt{n^2 -1}}{3(n^2-2)} \leq   \sqrt{\frac{ 8(n+2)}{ (n+1) (2 n^2 + 4n +3 ) }}
$$
for $n \geq 2$. Consequently, we get from Theorem \ref{dtm3} and Theorem \ref{ctm2} that

\begin{cor}
Let $(M, HM , J_b, \theta)$ be a simply connected Sasakian manifold of dimension $2n +1 \geq 5$ and with constant positive pseudo-Hermitian scalar curvature. Assume that
\begin{align*}
\sup_M \big( |C| + \sqrt{2} |E| \big) <\frac{2\sqrt{n^2 -1}}{3(n^2-2)} \: \rho .
\end{align*}
Then $(M, HM , J_b, \theta)$ is D-homothetic to $S^{2n+1}$.
\end{cor}

\section*{Acknowledgement}
We are grateful to Professor Jih-Hsin Cheng for his interesting lectures which drew our attention to Chern-Moser theory.

\bibliographystyle{plain}

\bibliography{gap}

Yuxin Dong

\emph{School of Mathematical Sciences}

\emph{Fudan University, Shanghai 200433 }

\emph{P. R. China}

yxdong@fudan.edu.cn

\vspace{12pt}

Hezi Lin

\emph{School of Mathematics and Computer Science}

\emph{Fujian Normal University, Fuzhou 350108}

\emph{P. R. China}

lhz1@fjnu.edu.cn.

\vspace{12pt}

Yibin Ren

\emph{School of Mathematical Sciences}

\emph{Fudan University, Shanghai 200433 }

\emph{P. R. China}

allenrybqqm@hotmail.com

\end{document}